\documentclass[11pt,reqno]{amsart}

\usepackage[T1]{fontenc}
\usepackage{lmodern}
\usepackage{microtype}
\usepackage{amsmath,amssymb,amsthm,mathtools}
\usepackage{mathrsfs}
\usepackage{booktabs,tabularx,array}
\usepackage{enumitem}
\usepackage{xcolor}
\usepackage[hidelinks]{hyperref}
\usepackage{hyperref}

\newtheorem{theorem}{Theorem}[section]
\newtheorem{proposition}[theorem]{Proposition}
\newtheorem{lemma}[theorem]{Lemma}
\newtheorem{corollary}[theorem]{Corollary}

\theoremstyle{definition}
\newtheorem{definition}[theorem]{Definition}

\theoremstyle{remark}
\newtheorem{remark}[theorem]{Remark}

\newcommand{\Ric}{\operatorname{Ric}}
\newcommand{\Rm}{\operatorname{Rm}}
\newcommand{\Scal}{\operatorname{Scal}}
\newcommand{\Ein}{\operatorname{Ein}}

\newcommand{\tr}{\operatorname{tr}}
\newcommand{\diver}{\operatorname{div}}
\newcommand{\Hess}{\operatorname{Hess}}

\newcommand{\dist}{\operatorname{dist}}

\newcommand{\dd}{\,\mathrm d}
\newcommand{\R}{\mathbb R}
\newcommand{\T}{\mathbb T}
\newcommand{\Sph}{\mathbb S}
\newcommand{\cP}{\mathcal P}
\newcommand{\cN}{\mathcal N}
\newcommand{\cZ}{\mathcal Z}
\newcommand{\dsec}{\delta_{\mathrm{sec}}}
\newcommand{\dric}{\delta_{\mathrm{Ric}}}
\newcommand{\abs}[1]{\lvert #1\rvert}
\newcommand{\norm}[1]{\lVert #1\rVert}
\newcommand{\inner}[2]{\langle #1,#2\rangle}
\newcommand{\eps}{\varepsilon}

\title[Curvature sign rigidity and sharp pinching thresholds]
{Curvature Sign Rigidity and Sharp Pointwise Pinching Thresholds}
\author{Minbo Gao$^1$}
\author{Yuhang Liu$^{*2}$}
\author{Genyuan Zhang$^3$}
\email{yuhang.liu02@xjtlu.edu.cn}
\date{2026-9-4}


\hypersetup{
  pdftitle={Curvature Sign Rigidity and Sharp Pointwise Pinching Thresholds},
  pdfauthor={Minbo Gao, Yuhang Liu, Genyuan Zhang},
  pdfsubject={Sectional and Ricci curvature sign rigidity under pointwise pinching},
  pdfkeywords={sectional curvature, Ricci curvature, pointwise pinching, Einstein tensor, divergence-free tensor, sign rigidity}}
  
\begin{document}
\footnotetext[1]{Institute of Software, Chinese Academy of Sciences}
\footnotetext[2]{Department of Applied Mathematics, Xi'an Jiaotong-Liverpool University}
\footnotetext[3]{Department of Applied Mathematics, Xi'an Jiaotong-Liverpool University}

\maketitle

\begin{abstract}
We study connected Riemannian manifolds on which either the sectional curvature or the Ricci tensor is, at each point, strictly positive, strictly negative, or zero, and ask whether the two signs can coexist under pointwise pinching.  A general support-rigidity theorem for positive semidefinite divergence-free symmetric tensors is the common analytic mechanism.

For sectional curvature in dimension $n\ge 3$, any locally uniform positive lower bound for the absolute pointwise pinching ratio rules out a change of sign.  No completeness, compactness, curvature bound, or regularity of the flat interface is required.  With a fixed pinching constant $\delta_0$, the flat set has no $C^1$ hypersurface piece; it is empty when $\delta_0>1/2$, and is locally porous when $\delta_0=1/2$.  These conclusions are sharp in several senses: there are smooth conformally flat local metrics whose sectional curvature changes sign across a flat hypersurface when the pinching degenerates, and there are closed exactly $1/q$-pinched metrics on spheres with isolated flat points.

For Ricci curvature, the sharp threshold is
\[
\delta_c=\frac1{n-1}.
\]
A locally uniform gap $\dric>\delta_c$ forces one Ricci sign globally.  Conversely, for every $0<\delta<\delta_c$ there are local sign-changing metrics with exact pointwise pinching $\dric\equiv\delta$, and exact local examples also exist at $\delta=\delta_c$.  Pointwise strictness $\dric>\delta_c$ is insufficient if the gap collapses at a Ricci-flat interface.  For every subcritical $\delta$ we give explicit closed metrics on $\Sph^1\times\Sph^{n-1}$, and complete periodic lifts to $\R\times\Sph^{n-1}$, whose optimal global lower pinching constant is exactly $\delta$.  Their pinching is not pointwise constant.  Finally, we prove two ansatz-specific obstructions to compactifying the exact local constructions. The main content of the proof is generated by ChatGPT 5.6 sol and verified by the authors.
\end{abstract}

\section{Introduction and main results}

Pointwise pinching usually enters Riemannian geometry after a curvature sign has already been fixed.  The differentiable quarter-pinching theorem and its weak endpoint classification are prominent examples: compact manifolds with positive or nonnegative pointwise quarter-pinched sectional curvature are strongly constrained by the work of Brendle and Schoen \cite{BrendleSchoenActa,BrendleSchoenJAMS}; related flag-pinching results were obtained by Ni and Wilking \cite{NiWilking}.  Here the question is logically prior to classification: if every point is strictly positive, strictly negative, or flat, can the sign itself vary on a connected manifold?

The answer is radically different for sectional and Ricci curvature.  For sectional curvature, every locally uniform positive pinching bound forces sign rigidity.  For Ricci curvature, sign change is possible up to the sharp threshold $1/(n-1)$ and impossible above it, provided that the supercritical gap is locally uniform.  Both statements are controlled by the Einstein tensor, but the relevant algebra is different.

Throughout we use the curvature convention for which the unit round sphere has positive sectional curvature.  The Einstein tensor is
\begin{equation}\label{eq:Ein-def}
\Ein=\Ric-\frac12\Scal\,g,
\end{equation}
and satisfies the contracted second Bianchi identity
\begin{equation}\label{eq:Bianchi}
\diver\Ein=0;
\end{equation}
see, for example, \cite[Chapters 7 and 8]{Lee} or \cite[Chapter 1]{Besse}.

\subsection{The two sign trichotomies}

\begin{definition}[Sectional sign trichotomy]
A Riemannian manifold $(M^n,g)$ satisfies the sectional sign trichotomy if at every point $p$ exactly one of the following holds:
\begin{enumerate}[label=(\roman*)]
\item $K_p(\sigma)>0$ for every $2$-plane $\sigma\subset T_pM$;
\item $K_p(\sigma)<0$ for every $2$-plane $\sigma\subset T_pM$;
\item $\Rm_p=0$.
\end{enumerate}
At a nonflat point define
\[
\dsec(p)=\frac{\min_{\sigma\subset T_pM}\abs{K_p(\sigma)}}
{\max_{\sigma\subset T_pM}\abs{K_p(\sigma)}}\in(0,1].
\]
We say that the sectional pinching is \emph{locally uniform} if for every compact $K\subset M$ there is $\delta_K>0$ such that $\dsec\ge\delta_K$ on the nonflat part of $K$.
\end{definition}

\begin{definition}[Ricci sign trichotomy]
A Riemannian manifold satisfies the Ricci sign trichotomy if at every point $p$ the Ricci tensor is positive definite, negative definite, or zero.  At a non-Ricci-flat point let $\lambda_1(p),\dots,\lambda_n(p)$ be the Ricci eigenvalues and define
\[
\dric(p)=\frac{\min_i\abs{\lambda_i(p)}}{\max_i\abs{\lambda_i(p)}}\in(0,1].
\]
For $n\ge3$ set
\begin{equation}\label{eq:critical}
\delta_c=\frac1{n-1}.
\end{equation}
We say that the Ricci pinching is \emph{locally uniformly supercritical} if for every compact $K\subset M$ there is $\eps_K>0$ such that
\[
\dric\ge \delta_c+\eps_K
\]
on every non-Ricci-flat point of $K$.
\end{definition}

The distinction between exact pinching and a lower pinching bound is essential.  Exact pointwise pinching means $\dric\equiv\delta$ on the non-Ricci-flat set.  A uniform lower bound means only $\dric\ge\delta$.

\subsection{A common tensor theorem}

The analytic core is a support theorem for positive semidefinite divergence-free tensors.  Its Euclidean virial identity is closely related to the mean-stress identities used in continuum mechanics and to the first step of stationary-varifold monotonicity \cite{Allard}; positive divergence-free tensors have been studied systematically by Serre \cite{Serre2018,Serre2022}.

\begin{theorem}[Support rigidity]\label{thm:support-intro}
Let $(M,g)$ be connected and let $T\in L^1_{\mathrm{loc}}(\operatorname{Sym}^2T^*M)$ satisfy
\[
T\ge0,
\qquad
\int_M\inner{T}{\Hess\phi}\dd V_g=0
\quad\text{for every }\phi\in C_c^\infty(M).
\]
Assume that
\begin{equation}\label{eq:trace-ell-intro}
T\ge \frac1\kappa(\tr_gT)g
\end{equation}
almost everywhere, where $\kappa:M\to[n,\infty)$ is bounded on compact sets.  If $T=0$ almost everywhere on one nonempty open set, then $T=0$ almost everywhere on all of $M$.
\end{theorem}

The Hessian-test identity is implied by distributional divergence-freeness.  A positive spectral truncation then gives a general sign theorem.

\begin{theorem}[Sign rigidity for conditioned divergence-free tensors]\label{thm:tensor-intro}
Let $(M,g)$ be connected and let $E\in C^1(\operatorname{Sym}^2T^*M)$ satisfy $\diver E=0$.  Suppose that at every point $E$ is positive definite, negative definite, or zero.  Assume that on every compact set the absolute condition number of $E$ is bounded at its nonzero points.  Then $E$ cannot assume both definite signs.
\end{theorem}

The local boundedness in this statement is stronger than a merely pointwise finite condition number and weaker than one global bound.  The proof is given in Section~\ref{sec:tensors}.

\subsection{Sectional-curvature results}

Let $\cP_{\mathrm{sec}}$, $\cN_{\mathrm{sec}}$, and $\cZ_{\mathrm{sec}}$ denote the positive, negative, and flat sets for sectional curvature.

\begin{theorem}[Sectional sign rigidity]\label{thm:sectional-intro}
Let $(M^n,g)$ be connected, $n\ge3$, and satisfy the sectional sign trichotomy.  If $\dsec$ is locally uniformly bounded away from zero, then exactly one of the following holds:
\begin{enumerate}[label=(\alph*)]
\item $\Rm\equiv0$;
\item $\cN_{\mathrm{sec}}=\varnothing$, $\cP_{\mathrm{sec}}\ne\varnothing$, and $\cZ_{\mathrm{sec}}$ has empty interior;
\item $\cP_{\mathrm{sec}}=\varnothing$, $\cN_{\mathrm{sec}}\ne\varnothing$, and $\cZ_{\mathrm{sec}}$ has empty interior.
\end{enumerate}
No completeness, compactness, curvature bound, or regularity of $\partial\cZ_{\mathrm{sec}}$ is assumed.  More generally, it is enough to impose locally uniform pinching on $\cP_{\mathrm{sec}}$ alone: then either $\cP_{\mathrm{sec}}=\varnothing$, or $\cN_{\mathrm{sec}}=\varnothing$ and $\cZ_{\mathrm{sec}}$ has empty interior.  The symmetric one-sided statement holds with $\cP_{\mathrm{sec}}$ and $\cN_{\mathrm{sec}}$ interchanged.
\end{theorem}

The next theorem gives stronger information when one fixed $\delta_0$ works globally.

\begin{theorem}[Structure of the sectional flat set]\label{thm:flat-intro}
Under the hypotheses of Theorem~\ref{thm:sectional-intro}, assume that $g$ is not flat and that $\dsec\ge\delta_0>0$ at every nonflat point.  Then:
\begin{enumerate}[label=(\roman*)]
\item $\cZ_{\mathrm{sec}}$ contains no $C^1$ immersed hypersurface piece.  More generally, in a sufficiently small normal chart it contains no Lipschitz hypersurface graph whose Lipschitz constant is below a positive threshold depending only on the local geometry and the pinching bound.
\item If $\delta_0>1/2$, then $\cZ_{\mathrm{sec}}=\varnothing$.
\item If $\delta_0=1/2$, then $\cZ_{\mathrm{sec}}$ is locally porous.  Consequently it has zero Riemannian volume, and every relatively compact subset has upper Minkowski dimension strictly less than $n$.
\end{enumerate}
\end{theorem}

The strict threshold in part (ii) is optimal: Section~\ref{sec:sectional-sharpness} gives exactly $1/2$-pinched metrics on $\Sph^n$ that are positively curved away from two flat poles.  Uniform pinching itself is indispensable.  We construct a conformally flat metric whose sectional curvature is negative on one side of a flat hypersurface and positive on the other, with
\[
\dsec=e^{-1/d+o(1)}
\quad\text{on the negative side},
\qquad
\dsec=(4+o(1))d^4
\quad\text{on the positive side},
\]
where $d$ is distance to the interface.

\subsection{Ricci-curvature results}

Let $\cP_{\Ric}$, $\cN_{\Ric}$, and $\cZ_{\Ric}$ denote the positive, negative, and Ricci-flat sets.

\begin{theorem}[Sharp Ricci threshold]\label{thm:ricci-intro}
Let $(M^n,g)$ be connected, $n\ge3$, and satisfy the Ricci sign trichotomy.
\begin{enumerate}[label=(\alph*)]
\item If the Ricci pinching is locally uniformly supercritical, then $\cP_{\Ric}$ and $\cN_{\Ric}$ cannot both be nonempty.  In particular, the conclusion holds under one global bound
\[
\dric\ge\delta_0>\frac1{n-1}.
\]
More generally, a locally uniform supercritical bound on $\cP_{\Ric}$ alone implies that either $\cP_{\Ric}=\varnothing$, or $\cN_{\Ric}=\varnothing$ and $\cZ_{\Ric}$ has empty interior; symmetrically for $\cN_{\Ric}$.
\item For every $0<\delta<1/(n-1)$ there is a smooth metric on $(-\eps,\eps)\times\T^{n-1}$ such that
\[
\Ric>0\ (t<0),\qquad \Ric=0\ (t=0),\qquad \Ric<0\ (t>0),
\]
and $\dric\equiv\delta$ at every $t\ne0$.
\item There is a connected smooth local example with nonempty positive, negative, and Ricci-flat regions and
\[
\dric\equiv\frac1{n-1}
\]
at every non-Ricci-flat point.
\item There are connected local examples with both Ricci signs such that
\[
\dric(p)>\frac1{n-1}
\]
at every non-Ricci-flat point, while the infimum of $\dric$ is exactly $1/(n-1)$.
\end{enumerate}
\end{theorem}

Thus the threshold and the need for a locally uniform gap are both sharp.  The rigidity statement requires neither completeness nor compactness.

For the weaker global lower-bound problem there are explicit compact examples.

\begin{theorem}[Closed and complete subcritical lower-bound examples]\label{thm:global-intro}
Let $n\ge3$, put $m=n-1$, and let $0<\delta<1/m$.  Define
\begin{equation}\label{eq:a-delta-intro}
 a_\delta=\frac{m(1-\delta)}{1-m\delta}.
\end{equation}
Then the metric
\begin{equation}\label{eq:compact-family-intro}
 g_\delta=\dd t^2+(a_\delta-\cos t)^2g_{\Sph^m}
\end{equation}
on $\Sph^1\times\Sph^m$ satisfies the Ricci sign trichotomy, has both Ricci signs, and obeys
\[
\dric\ge\delta
\]
at every non-Ricci-flat point, with equality somewhere.  Its optimal global lower pinching constant is therefore exactly $\delta$.  The lift to $\R\times\Sph^m$ is complete and has the same properties.
\end{theorem}

The pinching in \eqref{eq:compact-family-intro} varies with $t$ and tends to $1/m$ at the Ricci-flat hypersurfaces.  Hence these are not exact pointwise-$\delta$ examples.  We also prove that two direct compactification schemes for the exact local metrics fail: an exact subcritical or critical phase cannot emerge from a Ricci-flat boundary in the singly warped space-form class, and a fixed labeled extremal relation in the periodic doubly warped class forces flatness.

In preparation of this paper, we had several rounds of conversations with ChatGPT 5.6 sol, and obtained results of increasing generality. At the end of the conversation, we asked ChatGPT to organize all results into a single paper, and then we reviewed the generated paper and revised it wherever necessary.

\textbf{Acknowledgement:} we thank Felix Ye and Ziheng Zou for using AI agents to give independent proofs, which have been absorbed into this paper. We also thank Zipei Nie for using ChatGPT to provide partial results.

\section{Divergence-free tensors and support rigidity}\label{sec:tensors}

This section proves Theorems~\ref{thm:support-intro} and~\ref{thm:tensor-intro}.  The support theorem requires only the Hessian-test identity, not the full first-order divergence equation.  In Euclidean coordinates it is a double-divergence or adjoint equation.  Writing $T=\rho A$, where $\rho=\tr T$ and $A=T/\tr T$, connects it with the theory of nonnegative adjoint solutions for uniformly elliptic nondivergence operators; see \cite{Bauman,FabesStroock,DongEscauriazaKim}.  The proof below is instead a direct virial argument.

\subsection{A local virial monotonicity formula}

\begin{lemma}[Quadratic-distance Hessian estimate]\label{lem:hess-distance}
For every $y\in M$ there are $r_0>0$ and $C_0<\infty$ such that $B_{r_0}(y)$ is a normal ball and, with $\rho=d(y,\cdot)$ and $u=\rho^2/2$,
\begin{equation}\label{eq:hess-error}
\norm{\Hess u-g}_{\mathrm{op}}\le C_0\rho^2
\qquad\text{on }B_{r_0}(y).
\end{equation}
The constants may be chosen uniformly for $y$ in a relatively compact set.
\end{lemma}

\begin{proof}
In geodesic normal coordinates centered at $y$, one has $u(x)=\abs{x}^2/2$, while $g_{ij}=\delta_{ij}+O(\abs{x}^2)$ and the Christoffel symbols are $O(\abs{x})$.  Hence $\nabla^2u=g+O(\rho^2)$.  Uniformity on compact sets follows from a uniformly normal neighborhood and compactness; see \cite[Chapter 6]{Lee} or \cite[Chapter 3]{doCarmo}.
\end{proof}

\begin{lemma}[Virial identity and monotonicity]\label{lem:virial}
Let $T\in L^1(B_{r_0}(y);\operatorname{Sym}^2T^*M)$ be positive semidefinite and satisfy
\[
\int\inner{T}{\Hess\phi}\dd V=0
\quad\text{for every }\phi\in C_c^\infty(B_{r_0}(y)).
\]
Assume $T\ge\bar\kappa^{-1}(\tr T)g$ almost everywhere for some constant $\bar\kappa\ge n$.  Put
\[
m(R)=\int_{B_R(y)}\tr T\dd V.
\]
After decreasing $r_0$ if necessary, there is $C_0$ as in Lemma~\ref{lem:hess-distance} such that
\begin{equation}\label{eq:monotone-quantity}
R\longmapsto R^{-\bar\kappa}
\exp\!\left(-\frac{\bar\kappa C_0R^2}{2}\right)m(R)
\end{equation}
is nonincreasing on $(0,r_0)$.
\end{lemma}

\begin{proof}
Let $\rho=d(y,\cdot)$ and $u=\rho^2/2$.  For almost every $R$ define
\[
\tau(R)=\int_{\partial B_R(y)}\tr T\dd A,
\qquad
\sigma(R)=\int_{\partial B_R(y)}T(\nabla\rho,\nabla\rho)\dd A,
\]
and
\[
E(R)=\int_{\partial B_R(y)}\inner{T}{\Hess u-g}\dd A.
\]
By positivity, $0\le\sigma\le\tau$, and \eqref{eq:hess-error} gives
\begin{equation}\label{eq:E-bound}
\abs{E(R)}\le C_0R^2\tau(R)
\end{equation}
for almost every $R$.  The coarea formula implies that $m$ is absolutely continuous and $m'=\tau$ almost everywhere.

Fix $\chi\in C_c^\infty((0,r_0))$ and define
\[
\Psi(s)=-\int_s^{r_0}t\chi(t)\dd t,
\qquad
\phi=\Psi\circ\rho.
\]
The function $\phi$ is smooth: it is constant near $y$ and vanishes near $\partial B_{r_0}(y)$.  Since $\Psi'(s)=s\chi(s)$,
\[
\Hess\phi=\rho\chi'(\rho)\dd\rho\otimes\dd\rho+\chi(\rho)\Hess u.
\]
Testing against $T$ and using coarea yields
\[
\int_0^{r_0}\chi'(R)R\sigma(R)\dd R
 +\int_0^{r_0}\chi(R)\bigl(\tau(R)+E(R)\bigr)\dd R=0.
\]
Thus, in distributions on $(0,r_0)$,
\begin{equation}\label{eq:Fprime}
\bigl(R\sigma(R)\bigr)'=\tau(R)+E(R).
\end{equation}
For every $R<r_0$, \eqref{eq:E-bound} gives
\[
\abs{\tau+E}\le(1+C_0r_0^2)\tau\in L^1(0,R).
\]
Hence $r\sigma(r)$ has an absolutely continuous representative on $[0,R]$ and a finite limit at zero.  Moreover,
\[
\int_0^R\frac{s\sigma(s)}s\dd s=\int_0^R\sigma(s)\dd s\le m(R)<\infty.
\]
Since $s\sigma(s)\ge0$, that limit must be zero.  Integrating \eqref{eq:Fprime} gives
\begin{equation}\label{eq:virial-identity}
R\sigma(R)=m(R)+\int_0^RE(s)\dd s
\end{equation}
for almost every $R$.  By \eqref{eq:E-bound},
\begin{equation}\label{eq:virial-error}
\abs{\int_0^RE(s)\dd s}
\le C_0R^2m(R).
\end{equation}

The trace ellipticity gives
\[
\sigma(R)\ge\frac1{\bar\kappa}\tau(R)
=\frac1{\bar\kappa}m'(R)
\]
for almost every $R$.  Combining this with the upper bound in \eqref{eq:virial-identity}--\eqref{eq:virial-error},
\[
\frac R{\bar\kappa}m'(R)
\le(1+C_0R^2)m(R).
\]
Equivalently,
\[
\frac{\dd}{\dd R}
\left[R^{-\bar\kappa}e^{-\bar\kappa C_0R^2/2}m(R)\right]\le0
\]
for almost every $R$, proving the claim.
\end{proof}

\begin{corollary}[Radially integrable loss of conditioning]\label{cor:radial-conditioning}
In the setting of Lemma~\ref{lem:virial}, replace the constant $\bar\kappa$ by a measurable function $\kappa\ge n$ and assume
\[
T\ge\frac1\kappa(\tr T)g.
\]
For almost every $r$, let
\[
\kappa_y(r)=\operatorname*{ess\,sup}_{\partial B_r(y)}\kappa.
\]
If $T=0$ almost everywhere on $B_\eta(y)$ and, for some $R\in(\eta,r_0)$,
\begin{equation}\label{eq:radial-integrability}
\int_\eta^R\kappa_y(r)\left(\frac1r+C_0r\right)\dd r<\infty,
\end{equation}
then $T=0$ almost everywhere on $B_R(y)$.
\end{corollary}

\begin{proof}
The trace-ellipticity estimate on $\partial B_r(y)$ gives $\sigma(r)\ge m'(r)/\kappa_y(r)$ for almost every $r$.  The virial identity and error estimate therefore yield
\[
m'(r)\le\kappa_y(r)\left(\frac1r+C_0r\right)m(r).
\]
Gronwall's inequality on $[\eta,R]$, using $m(\eta)=0$ and \eqref{eq:radial-integrability}, gives $m(R)=0$.
\end{proof}

In Euclidean space $C_0=0$ and \eqref{eq:virial-identity} is the exact mean-stress identity
\[
\int_{B_R(y)}\tr T\dd x
=R\int_{\partial B_R(y)}T(\nu,\nu)\dd A.
\]

\subsection{Proof of support rigidity}

\begin{proof}[Proof of Theorem~\ref{thm:support-intro}]
Choose a geodesic ball $B_\eta(y)$ contained in the given open zero set.  Take a normal ball $B_{r_0}(y)$ as in Lemma~\ref{lem:virial}, with $\eta<r_0$, and let
\[
\bar\kappa=\operatorname*{ess\,sup}_{B_{r_0}(y)}\kappa<\infty.
\]
Then \eqref{eq:trace-ell-intro} implies the constant trace-ellipticity bound with $\bar\kappa$.  Since $m(\eta)=0$, the nonnegative monotone quantity \eqref{eq:monotone-quantity} vanishes for every $R\in[\eta,r_0)$.  Hence $m(R)=0$, and positivity gives $T=0$ almost everywhere on $B_{r_0}(y)$.

This local propagation can be iterated along any path.  More precisely, join an arbitrary point to the original zero ball by a compact path and cover the path by finitely many normal balls, chosen so that the center of each new ball lies in the zero region already obtained.  At each step a small zero ball around the new center expands to the full normal ball by the preceding argument.  Connectedness therefore gives $T=0$ almost everywhere on all of $M$.
\end{proof}

\begin{remark}[Sharp growth exponent]\label{rem:sharp-exponent}
On $\R^n\setminus\{0\}$, for $a\ge0$, let
\[
T_a=r^a\,\dd r^2+\frac{n-1+a}{n-1}r^a\bigl(g_{\mathrm{Euc}}-\dd r^2\bigr).
\]
A direct radial-divergence calculation gives $\diver T_a=0$.  Moreover
\[
T_a\ge\frac1{n+a}(\tr T_a)g_{\mathrm{Euc}},
\qquad
\int_{B_R}\tr T_a\dd x\asymp R^{n+a}.
\]
Thus the exponent $\bar\kappa$ in Lemma~\ref{lem:virial} is optimal.  If $a=3(n-1)$, then the tensor condition number is exactly $4$; when $n$ is odd, the Cartesian expression
\[
T_a=4r^a g_{\mathrm{Euc}}-3r^{a-2}x\otimes x
\]
is polynomial and extends smoothly across the origin, vanishing only there.
\end{remark}

\begin{remark}[Linear degeneration at the tensor level]\label{rem:linear-tensor}
The local boundedness of the conditioning coefficient cannot be replaced by pointwise finiteness.  On $\R^n$, choose $\gamma>0$ with $(n-1)\gamma>1$ and, for $r>1$, set
\[
a(r)=r^{1-n}\left(\frac{r-1}{r}\right)^{(n-1)\gamma},
\qquad
b(r)=\frac{\gamma}{r-1}a(r),
\]
\[
T=a(r)\,\dd r^2+b(r)\bigl(g_{\mathrm{Euc}}-\dd r^2\bigr),
\]
extended by zero to $r\le1$.  The radial divergence equation
\[
a'+\frac{n-1}{r}(a-b)=0
\]
holds for $r>1$.  Writing $p=(n-1)\gamma>1$, one has $a=O((r-1)^p)$ and $b=O((r-1)^{p-1})$; their first derivatives are locally integrable across $r=1$.  Thus the zero extension lies in $W^{1,1}_{\mathrm{loc}}$ and is distributionally divergence-free.  It is continuous, has an open dead core, and is positive definite for $r>1$, while its condition number satisfies
\[
\operatorname{cond}T\sim\frac{\gamma}{r-1}
\qquad(r\downarrow1).
\]
Thus linear degeneration is already compatible with loss of support rigidity at the tensor level; the integral in \eqref{eq:radial-integrability} diverges logarithmically at the interface.
\end{remark}

\subsection{Positive spectral truncation}

For a self-adjoint endomorphism $A$, write $A_+$ for the endomorphism obtained by replacing each eigenvalue $\lambda$ by $\max\{\lambda,0\}$.

\begin{lemma}[Divergence-free positive part]\label{lem:positive-part}
Let $E\in C^1(\operatorname{Sym}^2T^*M)$ satisfy $\diver E=0$ and suppose that $E$ is positive definite, negative definite, or zero at every point.  Let $A=g^{-1}E$ and define
\[
H(X,Y)=g(A_+X,Y).
\]
Then $H\in W^{1,\infty}_{\mathrm{loc}}$, $H=E$ on $\{E>0\}$, $H=0$ on $\{E\le0\}$, and
\[
\diver H=0
\]
in the sense of distributions.
\end{lemma}

\begin{proof}
Work in a relatively compact coordinate chart and a smooth orthonormal trivialization.  The map from symmetric matrices to their positive parts is the metric projection onto the closed convex cone of positive semidefinite matrices, hence is nonexpansive; see \cite[Chapter 4]{BauschkeCombettes}.  Since $A$ is locally Lipschitz, so is $A_+$, and therefore $H\in W^{1,\infty}_{\mathrm{loc}}$.

Let $Z=\{E=0\}$.  For every scalar $C^1$ component $f$ of $E$, the gradient vanishes almost everywhere on the level set $\{f=0\}$; this is the standard Stampacchia property, see \cite[Section 4.2]{EvansGariepy}.  Since all components vanish on $Z$, one has $\nabla E=0$ almost everywhere on $Z$.  At such a point, $A(y)=o(d(x,y))$, and nonexpansiveness gives $A_+(y)=o(d(x,y))$, so $\nabla H=0$ almost everywhere on $Z$.

On the positive phase, $H=E$ and $\nabla H=\nabla E$; on the negative phase, $H=0$ and $\nabla H=0$.  The weak derivative of a locally Lipschitz tensor agrees with its almost-everywhere derivative.  Consequently
\[
\diver H=\mathbf 1_{\{E>0\}}\diver E=0
\]
almost everywhere and hence distributionally.  No singular measure is produced on $Z$.
\end{proof}

\begin{proof}[Proof of Theorem~\ref{thm:tensor-intro}]
Let $P=\{E>0\}$ and $N=\{E<0\}$.  Assume both are nonempty and let $H$ be the positive part from Lemma~\ref{lem:positive-part}.  On a compact set, suppose the condition number of $E$ is at most $C$.  If $0<\mu_{\min}\le\mu_i\le\mu_{\max}$ are the eigenvalues of $H$ on $P$, then
\[
\tr H\le\mu_{\min}+(n-1)\mu_{\max}
\le\bigl(1+(n-1)C\bigr)\mu_{\min}.
\]
Thus
\[
H\ge\frac1{1+(n-1)C}(\tr H)g,
\]
with a coefficient bounded on compact sets.  The tensor $H$ is positive semidefinite, distributionally divergence-free, and vanishes on the nonempty open set $N$.  Theorem~\ref{thm:support-intro} gives $H\equiv0$, contradicting $H=E>0$ on $P$.  Hence both signs cannot occur.
\end{proof}

\begin{remark}[Why semidefiniteness is insufficient]\label{rem:semidefinite-deadcore}
On a flat torus with coordinates $(x_1,\dots,x_n)$, let $\varphi(x_n)\ge0$ be smooth, nonzero, and identically zero on a nonempty interval.  Then
\[
T=\varphi(x_n)\sum_{i=1}^{n-1}\dd x_i^2
\]
is smooth, positive semidefinite, divergence-free, and has an open dead core.  The trace-ellipticity hypothesis fails because $T$ has a null direction wherever it is nonzero.  This example will explain the Ricci endpoint.
\end{remark}

\section{Sectional curvature}\label{sec:sectional}

\subsection{The Einstein tensor as a complementary curvature sum}

\begin{lemma}[Complementary sectional-curvature identity]\label{lem:Ein-section}
Let $v=e_1$ be a unit vector and extend it to an orthonormal basis $e_1,\dots,e_n$.  Then
\begin{equation}\label{eq:Ein-section}
\Ein(v,v)=-\sum_{2\le i<j\le n}K(e_i\wedge e_j).
\end{equation}
\end{lemma}

\begin{proof}
Since
\[
\Ric(v,v)=\sum_{j=2}^nK(e_1\wedge e_j),
\qquad
\frac12\Scal=\sum_{1\le i<j\le n}K(e_i\wedge e_j),
\]
subtraction gives \eqref{eq:Ein-section}.
\end{proof}

For $n\ge3$ the sum is nonempty.  Hence $-\Ein$ is positive definite at a point of positive sectional curvature, $-\Ein$ is negative definite at a point of negative sectional curvature, and $\Ein=0$ at a flat point.

The next estimate is sharper than the bound obtained by simply comparing every sum with the total number of terms.

\begin{lemma}[Sharp conditioning estimate]\label{lem:section-conditioning}
Assume the sectional sign trichotomy at a nonflat point $p$ and let $\delta=\dsec(p)$.  With the sign chosen so that $T=\mp\Ein$ is positive definite, one has
\begin{equation}\label{eq:Cndelta}
\frac{\lambda_{\max}(T)}{\lambda_{\min}(T)}
\le C_{n,\delta}:=
\frac{2+(n-3)\delta}{(n-1)\delta}.
\end{equation}
Consequently
\begin{equation}\label{eq:section-trace-ell}
T\ge\frac1{\kappa_{n,\delta}}(\tr T)g,
\qquad
\kappa_{n,\delta}=n-2+\frac2\delta.
\end{equation}
Both constants are sharp.
\end{lemma}

\begin{proof}
It suffices to work at a positive point and put $T=-\Ein$.  Let
\[
\Lambda=\max_\sigma K(\sigma),
\qquad
\lambda=\min_\sigma K(\sigma)=\delta\Lambda.
\]
If $T$ is a multiple of the metric there is nothing to prove.  Otherwise choose orthonormal eigenvectors $v=e_1$ and $w=e_2$ for the smallest and largest eigenvalues of $T$, and complete them to an orthonormal basis.  Put
\[
A=\sum_{3\le i<j\le n}K(e_i\wedge e_j).
\]
By Lemma~\ref{lem:Ein-section},
\[
T(v,v)=\sum_{j=3}^nK(e_2\wedge e_j)+A,
\qquad
T(w,w)=\sum_{j=3}^nK(e_1\wedge e_j)+A.
\]
The unshared sums have $n-2$ terms, while $A$ has $\binom{n-2}{2}$ terms.  Therefore
\begin{align*}
\frac{\lambda_{\max}(T)}{\lambda_{\min}(T)}
&\le
\frac{(n-2)\Lambda+\binom{n-2}{2}\delta\Lambda}
{(n-2)\delta\Lambda+\binom{n-2}{2}\delta\Lambda}\\
&=\frac{2+(n-3)\delta}{(n-1)\delta}.
\end{align*}
If $C=C_{n,\delta}$, then
\[
\tr T\le\lambda_{\min}+(n-1)\lambda_{\max}
\le\bigl(1+(n-1)C\bigr)\lambda_{\min},
\]
and $1+(n-1)C=n-2+2/\delta$, proving \eqref{eq:section-trace-ell}.

The constants are pointwise sharp for every $\delta\in(0,1]$.  Indeed, put $q=\delta^{-1}$ and consider a singly warped metric at a point where the radial sectional curvatures equal $qc$ and the tangential sectional curvatures equal $c>0$.  Such a curvature configuration is realized by prescribing the $2$-jet of a positive warping function.  The complementary-sum eigenvalues then have ratio
\[
\frac{2q+n-3}{n-1}=C_{n,\delta},
\]
and equality also holds in \eqref{eq:section-trace-ell}.  Proposition~\ref{prop:pinched-spheres} globalizes this equality configuration to a closed sphere whenever $q$ is an integer at least $2$.
\end{proof}

\subsection{Proof of sectional sign rigidity}

\begin{proof}[Proof of Theorem~\ref{thm:sectional-intro}]
Set $E=-\Ein$.  By \eqref{eq:Bianchi}, $E$ is smooth and divergence-free.  Lemma~\ref{lem:Ein-section} shows that $E$ is positive definite on $\cP_{\mathrm{sec}}$, negative definite on $\cN_{\mathrm{sec}}$, and zero on $\cZ_{\mathrm{sec}}$.

Assume first only that the pinching is locally uniform on $\cP_{\mathrm{sec}}$, and let $H$ be the positive spectral part of $E$.  Lemma~\ref{lem:positive-part} makes $H$ distributionally divergence-free.  Lemma~\ref{lem:section-conditioning} gives trace ellipticity with a coefficient bounded on compact sets, and $H$ vanishes on $\cN_{\mathrm{sec}}\cup\cZ_{\mathrm{sec}}$.  If $\cP_{\mathrm{sec}}\ne\varnothing$, then neither $\cN_{\mathrm{sec}}$ nor the interior of $\cZ_{\mathrm{sec}}$ can be nonempty, by Theorem~\ref{thm:support-intro}.  This proves the one-sided statement.  Applying the same argument to $-E$ proves its negative counterpart.  Under two-sided local uniform pinching, the three alternatives in the theorem follow immediately.
\end{proof}

\subsection{A hypersurface exclusion lemma}

We isolate the local flux argument needed for Theorem~\ref{thm:flat-intro}(i).

\begin{proposition}[No sufficiently flat zero hypersurface]\label{prop:no-graph}
Let $T\in C^1(\operatorname{Sym}^2T^*M)$ be nonzero, positive semidefinite, and divergence-free.  Suppose that wherever $T\ne0$ its eigenvalues satisfy
\begin{equation}\label{eq:cond-T}
\lambda_{\max}(T)\le C\lambda_{\min}(T)
\end{equation}
for one finite constant $C$.  Then, in a sufficiently small coordinate neighborhood, the zero set of $T$ contains no Lipschitz hypersurface graph with sufficiently small Lipschitz constant.  In particular it contains no $C^1$ immersed hypersurface piece.
\end{proposition}

\begin{proof}
Suppose a graph piece lies in the zero set.  Choose geodesic normal coordinates $x=(x^1,x')$ centered at a point of the graph, with the $x^1$-axis normal to its tangent hyperplane.  After shrinking the chart, write the graph as
\[
\Gamma=\{x^1=\gamma(x')\},
\qquad \gamma(0)=0,
\]
with arbitrarily small $\operatorname{Lip}\gamma$ in the $C^1$ case.

Raise both indices of $T$ and set
\[
W^i=\sqrt{\det g}\,T^{i1}.
\]
The metric is uniformly close to the Euclidean metric.  Thus \eqref{eq:cond-T} implies, after changing the constant, that the Euclidean eigenvalues $\mu_1\le\cdots\le\mu_n$ of the matrix $(T^{ij})$ satisfy $\mu_n\le C_*\mu_1$.  Positivity gives
\[
\sum_j(T^{1j})^2=(T^2)_{11}\le\mu_nT^{11},
\qquad T^{11}\ge\mu_1,
\]
so there is $c_*>0$ such that
\begin{equation}\label{eq:cone-row}
W^1\ge0,
\qquad
\abs{W'}\le c_*W^1.
\end{equation}
The divergence equation yields
\[
\partial_iW^i=-\sqrt{\det g}\,\Gamma^1_{ik}T^{ik}.
\]
Since $\Gamma=O(\abs{x})$ in normal coordinates and the entries of a positive matrix with bounded condition number are controlled by $T^{11}$, shrinking once more gives
\begin{equation}\label{eq:divW}
\abs{\partial_iW^i}\le A W^1
\end{equation}
for some finite $A$.

Choose $L<c_*^{-1}$ and assume $\operatorname{Lip}\gamma\le L$.  Let $\psi(x')$ be a downward cone with $\abs{\nabla\psi}\le c_*^{-1}$, positive above $\gamma$ near the origin and meeting it before the boundary of the coordinate disk.  Put
\[
U=\{x':\gamma(x')<\psi(x')\},
\qquad
F=\psi-\gamma,
\]
and for $t>0$ define
\[
\theta_t=\min\{\gamma+t,\psi\},
\qquad
\Omega_t=\{(x^1,x'):x'\in U,\ \gamma(x')<x^1<\theta_t(x')\}.
\]
The base flux vanishes because $T=0$ on $\Gamma$.  On the top graph the outward conormal is a positive multiple of $(1,-\nabla\theta_t)$.  Where $F>t$, \eqref{eq:cone-row} gives
\[
W^1-\nabla\gamma\cdot W'
\ge(1-c_*L)W^1,
\]
while on the conical cap $\theta_t=\psi$ the flux is nonnegative because $\abs{\nabla\psi}\le c_*^{-1}$.  The Gauss--Green formula for Lipschitz domains, as in \cite[Chapter 5]{EvansGariepy}, therefore implies
\[
(1-c_*L)\varphi(t)
\le\int_{\Omega_t}\diver_{\R^n}W\dd x
\le A\int_{\Omega_t}W^1\dd x,
\]
where
\[
\varphi(t)=\int_{\{F>t\}}W^1(\gamma(x')+t,x')\dd x'.
\]
Tonelli's theorem gives
\[
\int_{\Omega_t}W^1\dd x=\int_0^t\varphi(s)\dd s.
\]
Hence, with $\Phi(t)=\int_0^t\varphi(s)\dd s$,
\[
\Phi'(t)\le\frac{A}{1-c_*L}\Phi(t)
\quad\text{a.e.},
\qquad
\Phi(0)=0.
\]
Gronwall's inequality gives $\Phi\equiv0$.  Thus $W^1=0$ on a nonempty open subset.  At a nonzero point, \eqref{eq:cond-T} would force $T^{11}\ge\mu_1>0$, so $T=0$ on that open subset.  Theorem~\ref{thm:support-intro} then gives $T\equiv0$, contrary to hypothesis.
\end{proof}

\subsection{Growth from the flat set and porosity}

\begin{proposition}[Nondegeneracy at the zero set]\label{prop:nondegrowth}
Let $T$ be a nonzero smooth positive semidefinite divergence-free tensor on a connected manifold, and suppose
\[
T\ge\frac1\kappa(\tr T)g
\]
with one constant $\kappa\ge n$.  Then for every compact set $K\subset M$ there are $c>0$ and $R_0>0$ such that
\begin{equation}\label{eq:growth-T}
\int_{B_R(y)}\tr T\dd V\ge cR^\kappa
\end{equation}
for every $y\in K$ and $0<R<R_0$.
\end{proposition}

\begin{proof}
Choose $R_0$ uniformly so that Lemma~\ref{lem:virial} applies to balls centered in a neighborhood of $K$.  Since $T$ has no open zero set by Theorem~\ref{thm:support-intro},
\[
m_y(R_0)=\int_{B_{R_0}(y)}\tr T\dd V>0
\]
for every $y\in K$.  Continuity in $y$ and compactness give a uniform positive lower bound.  The monotonicity formula, used from $R$ up to $R_0$, then yields \eqref{eq:growth-T}.
\end{proof}

\begin{proof}[Proof of Theorem~\ref{thm:flat-intro}]
By Theorem~\ref{thm:sectional-intro}, one sectional sign occurs.  Put $T=-\Ein$ in the positive case and $T=\Ein$ in the negative case, so that $T\ge0$ globally.  Lemma~\ref{lem:section-conditioning} gives
\begin{equation}\label{eq:kappa-flat}
T\ge\frac1\kappa(\tr T)g,
\qquad
\kappa=n-2+\frac2{\delta_0}.
\end{equation}
Part (i) follows from Proposition~\ref{prop:no-graph}.

For parts (ii) and (iii), note that
\[
\tr T=\frac{n-2}{2}\Scal
\]
in the positive case; in the negative case use $-\Scal$.  Thus there is a smooth nonnegative function $s$ with $\tr T=(n-2)s/2$ and zero set $\cZ_{\mathrm{sec}}$.  Proposition~\ref{prop:nondegrowth}, together with the local upper volume bound, gives
\begin{equation}\label{eq:scalar-lower}
\sup_{B_R(y)}s\ge cR^\beta,
\qquad
\beta=\kappa-n=\frac{2(1-\delta_0)}{\delta_0},
\end{equation}
for $y$ in a compact subset of $\cZ_{\mathrm{sec}}$ and small $R$.

Because $s\ge0$ is smooth and vanishes on $\cZ_{\mathrm{sec}}$, its gradient also vanishes there.  A Taylor estimate on compact sets gives
\begin{equation}\label{eq:scalar-distance}
s(x)\le C\dist(x,\cZ_{\mathrm{sec}})^2.
\end{equation}
If $\delta_0>1/2$, then $\beta<2$, and \eqref{eq:scalar-lower} contradicts \eqref{eq:scalar-distance} as $R\downarrow0$.  Hence the flat set is empty.

If $\delta_0=1/2$, then $\beta=2$.  Apply \eqref{eq:scalar-lower} at scale $R/2$ and combine it with \eqref{eq:scalar-distance}.  One finds a point $x\in B_{R/2}(y)$ with
\[
\dist(x,\cZ_{\mathrm{sec}})\ge\sigma R
\]
for a locally uniform $\sigma>0$.  Therefore a ball of radius comparable to $R$ inside $B_R(y)$ misses $\cZ_{\mathrm{sec}}$.  This is local porosity.  Uniformly porous subsets of Euclidean space have upper Minkowski dimension strictly below the ambient dimension \cite{Salli}; applying this in finitely many coordinate charts gives the asserted measure and dimension conclusions.
\end{proof}

\section{Sharpness and examples for sectional curvature}\label{sec:sectional-sharpness}

\subsection{A sign-changing conformally flat metric with degenerating pinching}

For $g=e^{2u}g_{\mathrm{Euc}}$, the Schouten tensor relative to the Euclidean background is
\begin{equation}\label{eq:Schouten-conf}
S=-D^2u+\dd u\otimes\dd u-\frac12\abs{Du}^2g_{\mathrm{Euc}}.
\end{equation}
Up to the common positive conformal factor, sectional curvatures are sums of two eigenvalues of $S$; this is the standard conformal curvature formula, see \cite[Chapter 1]{Besse}.

\begin{proposition}[Degenerate-pinching sign change]\label{prop:conformal-signchange}
For every $n\ge3$ there is a connected smooth conformally flat metric on a slab in $\R^n$ satisfying the sectional sign trichotomy, with negative sectional curvature on one side of a flat hypersurface and positive sectional curvature on the other.  If $d$ denotes distance to the hypersurface, then
\begin{equation}\label{eq:degenerate-rates}
\dsec=e^{-1/d+o(1)}
\quad\text{on the negative side},
\qquad
\dsec=(4+o(1))d^4
\quad\text{on the positive side}.
\end{equation}
The asymptotics are uniform in the transverse variable on a fixed smaller slab.
\end{proposition}

\begin{proof}
Write $x=(t,z)\in\R\times\R^{n-1}$ and define the one-sided flat functions
\[
F(t)=\begin{cases}e^{1/t},&t<0,\\0,&t\ge0,\end{cases}
\qquad
H(t)=\begin{cases}e^{-1/t},&t>0,\\0,&t\le0.\end{cases}
\]
For small fixed $\eps>0$, set
\begin{equation}\label{eq:u-conformal}
u(t,z)=F(t)-\eps H(t)e^{\abs z^2},
\qquad
g=e^{2u}g_{\mathrm{Euc}}.
\end{equation}
We restrict to $-1/2<t<t_0$ and $\abs z<r_0$, with $t_0,r_0$ sufficiently small.

For $t<0$, $u=F(t)$ depends only on $t$.  The two types of eigenvalue sums in \eqref{eq:Schouten-conf} are
\[
-F''(t)
\quad\text{and}\quad
-(F'(t))^2.
\]
Here
\[
F'(t)=-\frac{e^{1/t}}{t^2},
\qquad
F''(t)=\frac{e^{1/t}(1+2t)}{t^4}>0
\quad(-1/2<t<0).
\]
Thus all sectional curvatures are negative.  Near $t=0$ the smaller magnitude is $(F')^2$ and the larger is $F''$, so
\[
\dsec(t,z)=\frac{(F')^2}{F''}
=\frac{e^{1/t}}{1+2t}
=e^{-1/\abs t+o(1)}.
\]

For $t>0$, put
\[
q=\eps e^{-1/t}e^{\abs z^2},
\qquad u=-q.
\]
Relative to the splitting $\R\partial_t\oplus\R^{n-1}$, the matrix of $S$ is
\[
S=\begin{pmatrix}a&b^T\\ b&C\end{pmatrix},
\]
where, with $s=\abs z^2$,
\begin{align*}
a&=\frac{q(1-2t)}{t^4}+\frac{q^2}{2t^4}-2q^2s,\\
b&=\frac{2q(1+q)}{t^2}z,\\
C&=\left(2q-\frac{q^2}{2t^4}-2q^2s\right)I
  +4q(1+q)z\otimes z.
\end{align*}
Because $q/t^4\to0$ as $t\downarrow0$, uniformly for $\abs z\le r_0$, one has
\[
\frac Cq\longrightarrow 2I+4z\otimes z,
\qquad
\frac{t^4a}{q}\longrightarrow1,
\qquad
\frac{t^2b}{q}\longrightarrow2z.
\]
The convergence is uniform for $\abs z\le r_0$, and therefore
\[
\frac{t^4}{q}\bigl(a-b^TC^{-1}b\bigr)
\longrightarrow
1-4z^T(2I+4z\otimes z)^{-1}z
=\frac1{1+2\abs z^2}>0.
\]
After decreasing $t_0$, the matrix $C$ and the Schur complement are positive.  Thus $S$ is positive definite and all sectional curvatures are positive.  Moreover,
\[
a^{-1}bb^T=4q\,z\otimes z+o(q),
\qquad
C-a^{-1}bb^T=2qI+o(q)
\]
uniformly for $\abs z\le r_0$.  Standard block perturbation therefore shows that the eigenvalues of $S$ consist of one eigenvalue
\[
\frac q{t^4}(1+o(1))
\]
and $n-1$ eigenvalues $2q(1+o(1))$, uniformly in $z$.  Hence the smallest sectional curvature is the sum of two transverse eigenvalues, $(4+o(1))q$, while the largest is $(1+o(1))q/t^4$.  This gives the second asymptotic in terms of the coordinate $t$.

All derivatives of $u$ vanish at $t=0$, so the metric is flat there and glues smoothly.  Moreover $u=o(1)$ uniformly in the transverse variable, and comparison with the Euclidean metric gives $d=\abs t(1+o(1))$ for the distance to the interface.  The two estimates above therefore give \eqref{eq:degenerate-rates} in terms of $d$.

\end{proof}

\subsection{Exactly pinched spheres with flat poles}

The standard warped-product curvature formulas used below go back to Bishop and O'Neill \cite{BishopONeill}.

\begin{proposition}[Exactly $1/q$-pinched spheres]\label{prop:pinched-spheres}
Let $n\ge3$, let $q\ge2$ be an integer, and let $C>0$.  There is a smooth rotationally symmetric metric on $\Sph^n$ that is positively curved away from two poles, flat at the poles, and exactly $1/q$-pinched at every nonflat point.
\end{proposition}

\begin{proof}
Let $f$ solve
\begin{equation}\label{eq:fq-ode}
f''=-qCf^{2q-1},
\qquad f(0)=0,
\qquad f'(0)=1.
\end{equation}
The first integral is
\begin{equation}\label{eq:fq-first}
(f')^2=1-Cf^{2q}.
\end{equation}
The solution increases from zero to $C^{-1/(2q)}$ and returns to zero in finite time with derivative $-1$.  Since the right-hand side of \eqref{eq:fq-ode} is a smooth odd function of $f$, the metric
\[
g=\dd t^2+f(t)^2g_{\Sph^{n-1}}
\]
extends smoothly over both endpoints to $\Sph^n$.

The radial and tangential sectional curvatures are
\[
K_{\mathrm{rad}}=-\frac{f''}{f}=qCf^{2q-2},
\qquad
K_{\mathrm{tan}}=\frac{1-(f')^2}{f^2}=Cf^{2q-2}.
\]
They are positive away from the poles and vanish at the poles because $q\ge2$.  The curvature operator is diagonal on radial and tangential two-forms, so every sectional curvature lies between $K_{\mathrm{tan}}$ and $K_{\mathrm{rad}}$.  Their ratio is exactly $q$, and hence $\dsec=1/q$.  These metrics also realize equality in Lemma~\ref{lem:section-conditioning}: the Einstein-tensor condition number is $C_{n,1/q}$.
\end{proof}

For $q=2$ this shows that the strict inequality $\delta_0>1/2$ in Theorem~\ref{thm:flat-intro}(ii) cannot be weakened.  For $q=4$ it gives an exactly quarter-pinched sphere with two flat poles.  Replacing \eqref{eq:fq-first} locally by $(f')^2=1+Cf^{2q}$ gives
\[
K_{\mathrm{rad}}=-qCf^{2q-2},
\qquad
K_{\mathrm{tan}}=-Cf^{2q-2},
\]
so isolated flat poles also occur on exact negatively curved local models.  For comparison, the non-spherical compact rank-one symmetric spaces with their standard metrics are classical exact quarter-pinched examples without flat points; see \cite{Besse}.  The present family shows that exact pinching neither forces local symmetry nor excludes isolated flat points.

\subsection{Dimension two}

When $n=2$ there is only one unoriented tangent $2$-plane, so $\dsec=1$ at every nonflat point.  The sign-rigidity theorem fails.  For example,
\begin{equation}\label{eq:surface-warped}
g=\dd t^2+(a-\cos t)^2\dd\theta^2,
\qquad a>1,
\end{equation}
on $\Sph^1\times\Sph^1$ has Gaussian curvature
\[
K=-\frac{\cos t}{a-\cos t},
\]
which changes sign.

\section{Ricci curvature: the algebraic threshold and rigidity}\label{sec:ricci-rigidity}

\subsection{The signed Einstein tensor}

\begin{lemma}[Einstein positivity above the Ricci threshold]\label{lem:ricci-Ein}
Let $n\ge3$.  Suppose $\Ric>0$ at a point, with eigenvalues
\[
0<m\le\lambda_i\le M,
\qquad m\ge\delta M.
\]
If $\delta>1/(n-1)$ and
\[
T=-\Ein=\frac12\Scal\,g-\Ric,
\]
then
\begin{equation}\label{eq:ricci-a}
T\ge a(\tr T)g,
\qquad
a=\frac{(n-1)\delta-1}{n(n-2)}>0.
\end{equation}
At a negative-Ricci point the same estimate holds for $T=\Ein$, after replacing the Ricci eigenvalues by their absolute values.
\end{lemma}

\begin{proof}
In a Ricci eigenbasis, the eigenvalues of $T=-\Ein$ are
\[
\tau_i=\frac12\left(\sum_{j\ne i}\lambda_j-\lambda_i\right).
\]
Therefore
\[
\tau_i\ge\frac12\bigl((n-1)m-M\bigr)
\ge\frac12\bigl((n-1)\delta-1\bigr)M.
\]
Also
\[
\tr T=\frac{n-2}{2}\Scal
\le\frac{n(n-2)}2M.
\]
Combining the inequalities gives \eqref{eq:ricci-a}.  The negative case follows by applying the same calculation to $-\Ric$.
\end{proof}

At the endpoint $\delta=1/(n-1)$ the tensor is only semidefinite.  Equality occurs for the spectrum
\begin{equation}\label{eq:critical-spectrum}
\bigl((n-1)q,q,\dots,q\bigr),
\qquad q>0,
\end{equation}
for which one eigenvalue of $-\Ein$ is zero.  Remark~\ref{rem:semidefinite-deadcore} shows that semidefiniteness and divergence-freeness alone cannot support an open-dead-core theorem.

\begin{lemma}[Divergence-free Ricci phase truncation]\label{lem:ricci-truncation}
Assume the Ricci sign trichotomy and put
\[
P=\{\Scal>0\}=\{\Ric>0\},
\qquad
N=\{\Scal<0\}=\{\Ric<0\}.
\]
Define
\[
T_+=\begin{cases}-\Ein,&P,\\0,&M\setminus P,\end{cases}
\qquad
T_-=\begin{cases}\Ein,&N,\\0,&M\setminus N.\end{cases}
\]
Then $T_+$ and $T_-$ are continuous, locally integrable, and distributionally divergence-free.
\end{lemma}

\begin{proof}
We prove the assertion for $T_+$.  On $P$, every Ricci eigenvalue lies between $0$ and $\Scal$, so each eigenvalue of $\Ein$ has absolute value at most $\Scal/2$.  In particular,
\begin{equation}\label{eq:Ein-size-Ricci}
\norm{\Ein}_{\mathrm{op}}\le\frac12\Scal
\qquad\text{on }P.
\end{equation}
At a point where $\Scal=0$, the sign trichotomy forces $\Ric=0$, and hence $\Ein=0$.  Thus the extension $T_+$ is continuous.

Choose smooth functions $\chi_\eps:\R\to[0,1]$ with $\chi_\eps=0$ on $(-\infty,0]$, $\chi_\eps=1$ on $[\eps,\infty)$, and $\abs{\chi_\eps'}\le C/\eps$.  Set
\[
T_{+,\eps}=-\chi_\eps(\Scal)\Ein.
\]
Using the contracted Bianchi identity,
\[
\diver T_{+,\eps}
=-\chi_\eps'(\Scal)\Ein(\nabla\Scal,\cdot).
\]
By \eqref{eq:Ein-size-Ricci}, on every compact set
\[
\abs{\diver T_{+,\eps}}
\le C\abs{\nabla\Scal}\,\mathbf1_{\{0<\Scal<\eps\}}.
\]
The right-hand side tends to zero in $L^1_{\mathrm{loc}}$ by dominated convergence, while $T_{+,\eps}\to T_+$ in $L^1_{\mathrm{loc}}$.  Hence $\diver T_+=0$ distributionally.  The proof for $T_-$ is identical after replacing $\Scal$ by $-\Scal$.
\end{proof}

\subsection{Supercritical sign rigidity}

\begin{theorem}[Locally uniform supercritical rigidity]\label{thm:ricci-rigidity}
Let $(M^n,g)$ be connected, $n\ge3$, satisfy the Ricci sign trichotomy, and be locally uniformly supercritical.  Then the positive and negative Ricci phases cannot both occur.  More generally, it is enough to impose a locally uniform supercritical bound on the positive phase alone: then either that phase is empty, or the negative phase is empty and the Ricci-flat set has empty interior.  The symmetric statement holds for the negative phase.
\end{theorem}

\begin{proof}
Assume first that a locally uniform supercritical bound is available on the positive phase, and let $T_+$ be the tensor from Lemma~\ref{lem:ricci-truncation}.  Lemma~\ref{lem:ricci-Ein} gives, on every compact set,
\[
T_+\ge a_K(\tr T_+)g
\]
for some $a_K>0$; outside the positive phase both sides vanish.  Thus $T_+$ satisfies the hypotheses of Theorem~\ref{thm:support-intro}.  If the positive phase is nonempty, the tensor cannot vanish on either a nonempty negative phase or an open Ricci-flat region.  This proves the positive one-sided statement.  The tensor $T_-$ proves the negative one-sided statement.  When both phases satisfy the locally uniform supercritical bound, they therefore cannot coexist.
\end{proof}

This proves part (a) of Theorem~\ref{thm:ricci-intro}.  A fixed global lower bound $\dric\ge\delta_0>1/(n-1)$ is a special case.

\section{Exact local Ricci sign change}\label{sec:ricci-local}

\subsection{Multiply warped curvature formulas}

Put $k=n-2\ge1$.  On an interval times a flat $(k+1)$-torus, consider
\begin{equation}\label{eq:double-warp}
g=\dd t^2+e^{2u(t)}\dd x^2+e^{2v(t)}\sum_{\alpha=1}^k\dd y_\alpha^2.
\end{equation}
Write
\[
A=u',\qquad B=v',\qquad D_A=A'+A^2,\qquad D_B=B'+B^2.
\]
The standard multiply warped formulas \cite{BishopONeill} give the sectional curvatures
\begin{align}
K(e_0,e_1)&=-D_A,&
K(e_0,e_\alpha)&=-D_B,\label{eq:double-sec1}\\
K(e_1,e_\alpha)&=-AB,&
K(e_\alpha,e_\beta)&=-B^2\quad(\alpha\ne\beta),\label{eq:double-sec2}
\end{align}
where
\[
e_0=\partial_t,
\qquad e_1=e^{-u}\partial_x,
\qquad e_\alpha=e^{-v}\partial_{y_\alpha}.
\]
Thus the Ricci tensor is diagonal with three eigenvalue types:
\begin{align}
\rho_0&=-D_A-kD_B,\label{eq:rho0}\\
\rho_1&=-D_A-kAB,\label{eq:rho1}\\
\rho_B&=-(B'+AB+kB^2),\label{eq:rhoB}
\end{align}
where $\rho_B$ has multiplicity $k$.

\subsection{Exact subcritical examples}

\begin{theorem}[Exact local subcritical construction]\label{thm:subcritical-local}
Let $n\ge3$ and $0<\delta<1/(n-1)$.  Then for some $\eps>0$ there is a smooth metric on $(-\eps,\eps)\times\T^{n-1}$ such that
\[
\Ric>0\quad(t<0),
\qquad
\Ric=0\quad(t=0),
\qquad
\Ric<0\quad(t>0),
\]
and $\dric\equiv\delta$ at every $t\ne0$.
\end{theorem}

\begin{proof}
Set
\[
C=\delta^{-1},
\qquad
\alpha=\frac{k}{C-1}.
\]
The subcritical condition is $C>k+1$, so $0<\alpha<1$.  Impose
\begin{equation}\label{eq:rho-relation-sub}
\rho_0=C\rho_1.
\end{equation}
Using \eqref{eq:rho0}--\eqref{eq:rho1}, this is equivalent to
\begin{equation}\label{eq:DA-relation}
D_A=\alpha D_B-C\alpha AB.
\end{equation}
Choose
\[
B(t)=t^2
\]
and let $A$ solve the smooth initial-value problem
\begin{equation}\label{eq:A-sub-ode}
A'=\alpha(2t+t^4)-C\alpha At^2-A^2,
\qquad A(0)=0.
\end{equation}
A unique smooth solution exists near zero.  Integrating $u'=A$ and $v'=B$ gives a metric of the form \eqref{eq:double-warp}.

Equation \eqref{eq:DA-relation} yields the useful identity
\begin{equation}\label{eq:rho1-useful}
\rho_1=\alpha(AB-D_B).
\end{equation}
The ODE gives
\[
A(t)=\alpha t^2+O(t^5).
\]
Since $D_B=2t+t^4$, we obtain
\begin{align*}
\rho_1(t)&=-2\alpha t+O(t^4),\\
\rho_0(t)&=C\rho_1(t),\\
\rho_B(t)&=-2t+O(t^4).
\end{align*}
Therefore
\begin{equation}\label{eq:ratio-sub}
\frac{\rho_B(t)}{\rho_1(t)}\longrightarrow
\frac1\alpha=\frac{C-1}{k}
\qquad(t\to0,\ t\ne0).
\end{equation}
Because $C>k+1$,
\[
1<\frac{C-1}{k}<C.
\]
After shrinking $\eps$, one has
\[
1<\frac{\rho_B}{\rho_1}<C
\]
for $0<\abs t<\eps$.  All Ricci eigenvalues therefore have the sign of $-t$, and their absolute values are ordered as
\[
\abs{\rho_1}<\abs{\rho_B}<\abs{\rho_0}=C\abs{\rho_1}.
\]
At $t=0$, $A=B=A'=B'=0$, so $\Ric=0$.  Finally,
\[
\dric=\frac{\abs{\rho_1}}{\abs{\rho_0}}=\frac1C=\delta.
\]
\end{proof}

\begin{remark}[How the three-dimensional quarter-pinched model evades rigidity]\label{rem:three-dim-Ein}
For $n=3$ and $\delta=1/4$, write the Ricci spectrum as
\[
4q,\quad rq,\quad q,
\qquad 1<r<4.
\]
The preceding expansions give $r(t)=3+t^3+O(t^4)$.  The eigenvalues of $-\Ein$ are
\[
\frac{r-3}{2}q,
\qquad
\frac{5-r}{2}q,
\qquad
\frac{r+3}{2}q.
\]
On the positive-Ricci side $t<0$, one has $q>0$ and $r<3$ near the interface, so $-\Ein$ is indefinite.  This is the concrete algebraic failure that disappears above the threshold $1/(n-1)$.
\end{remark}

\subsection{Exact critical collars}

Let
\begin{equation}\label{eq:flat-b}
b(s)=\begin{cases}e^{-1/s^2},&s>0,\\0,&s\le0.\end{cases}
\end{equation}
Then
\begin{equation}\label{eq:b-dominates}
b'(s)=\frac2{s^3}b(s),
\qquad
\frac{b}{b'}\to0,
\qquad
\frac{b^2}{b'}\to0
\quad(s\downarrow0).
\end{equation}
Put $m=n-1=k+1$.  In the doubly warped ansatz impose
\begin{equation}\label{eq:critical-identity}
D_A=D_B-mAB,
\end{equation}
that is,
\begin{equation}\label{eq:A-critical-ode}
A'=B'+B^2-mAB-A^2.
\end{equation}
A direct substitution gives
\begin{equation}\label{eq:rho-critical-relation}
\rho_0=m\rho_1.
\end{equation}

\begin{proposition}[Positive and negative exact-critical collars]\label{prop:critical-collars}
Fix $a>0$.
\begin{enumerate}[label=(\roman*)]
\item If $B=-b$ and $A$ solves \eqref{eq:A-critical-ode} with $A(0)=a$, then for sufficiently small $s>0$ the corresponding metric has $\Ric>0$ and $\dric=1/m$.
\item If $B=b$ and $A$ solves the same equation with $A(0)=a$, then for sufficiently small $s>0$ the metric has $\Ric<0$ and $\dric=1/m$.
\item After a suitable normalization of $u$ and $v$, both metrics are Euclidean for $s\le0$.
\end{enumerate}
\end{proposition}

\begin{proof}
For the positive collar set
\[
\beta=b'-b^2,
\qquad
\gamma=Ab.
\]
By \eqref{eq:b-dominates}, $\gamma/\beta\to0$, $b^2/\beta\to0$, and $b^2/\gamma=b/A\to0$.  Equations \eqref{eq:rho0}--\eqref{eq:rhoB} and \eqref{eq:critical-identity} give
\[
\rho_1=\beta-\gamma=:q,
\qquad
\rho_0=mq,
\qquad
\rho_B=\beta+\gamma-(k-1)b^2.
\]
For small $s>0$, $q>0$ and
\begin{align*}
\rho_B-q&=2\gamma-(k-1)b^2>0,\\
mq-\rho_B&=kb'-b^2-(k+2)Ab>0.
\end{align*}
Hence
\[
0<q<\rho_B<mq,
\]
which proves the positive assertion.

For the negative collar put
\[
\beta=b'+b^2,
\qquad
\gamma=Ab,
\qquad
Q=\beta-\gamma.
\]
Then
\[
\rho_1=-Q,
\qquad
\rho_0=-mQ,
\qquad
\rho_B=-S,
\]
where
\[
S=\beta+\gamma+(k-1)b^2.
\]
For small $s>0$,
\begin{align*}
S-Q&=2\gamma+(k-1)b^2>0,\\
mQ-S&=kb'+b^2-(k+2)Ab>0.
\end{align*}
Thus $0<Q<S<mQ$, proving the negative assertion.

For $s\le0$, $B=0$ and \eqref{eq:A-critical-ode} reduces to $A'=-A^2$.  Writing $R=1/a$, one has
\[
A(s)=\frac1{R+s}.
\]
Choose $e^{u(0)}=R$ and $v(0)=0$.  Then
\[
e^{u(s)}=R+s,
\qquad e^{v(s)}=1,
\qquad s\le0,
\]
and
\begin{equation}\label{eq:flat-collar}
g=\dd s^2+(R+s)^2\dd\theta^2+\sum_{\alpha=1}^k\dd y_\alpha^2,
\end{equation}
which is Euclidean in polar coordinates on $\R^2\times\R^k$.  Since $b$ is flat at zero, the collar metrics agree with the Euclidean metric to infinite order.
\end{proof}

\begin{theorem}[Connected exact-critical example]\label{thm:critical-local}
For every $n\ge3$ there is a connected smooth Riemannian $n$-manifold with nonempty positive-Ricci, negative-Ricci, and Ricci-flat regions such that
\[
\dric\equiv\frac1{n-1}
\]
at every non-Ricci-flat point.
\end{theorem}

\begin{proof}
Use Euclidean coordinates on $\R^2\times\R^k$, polar coordinates $(r,\theta)$ on $\R^2$, and set $s=r-R$.  Choose disjoint nonempty open sets $U_+,U_-\subset\Sph^1\times\R^k$ with disjoint closures.  For small $\eta,\eps>0$, let
\[
M=\{r<R\}\cup
\{R-\eta<r<R+\eps,\ (\theta,y)\in U_+\cup U_-\}.
\]
This is open and connected.  Put the Euclidean metric on $\{r<R\}$, the positive collar of Proposition~\ref{prop:critical-collars} over $U_+$, and the negative collar over $U_-$.  Formula \eqref{eq:flat-collar} and infinite-order agreement at $s=0$ give a smooth metric.  The pinching and signs follow from Proposition~\ref{prop:critical-collars}.
\end{proof}

The example is deliberately local and incomplete.  It proves that the endpoint lower bound $\dric\ge1/(n-1)$ cannot imply sign rigidity on arbitrary connected manifolds.

\subsection{Pointwise strictness without a locally uniform gap}

For a singly warped metric
\begin{equation}\label{eq:singly-sphere}
g=\dd s^2+f(s)^2g_{\Sph^m},
\qquad m=n-1,
\end{equation}
the Ricci eigenvalues are
\begin{align}
\lambda_0&=-m\frac{f''}{f},\label{eq:single-lambda0}\\
\lambda_S&=-\frac{f''}{f}+(m-1)\frac{1-(f')^2}{f^2},\label{eq:single-lambdaS}
\end{align}
where $\lambda_S$ has multiplicity $m$.

\begin{theorem}[Strict supercritical inequality with collapsing gap]\label{thm:strict-no-gap}
For every $n\ge3$ there is a connected local metric satisfying the Ricci sign trichotomy, with both definite signs, such that
\[
\dric(p)>\frac1{n-1}
\]
at every non-Ricci-flat point, while
\[
\inf_{\Ric\ne0}\dric=\frac1{n-1}.
\]
\end{theorem}

\begin{proof}
Let $b$ be the flat function \eqref{eq:flat-b}.  Start from a Euclidean radial collar
\[
f(s)=R+s,
\qquad s\le0.
\]
On a positive collar $s>0$, solve
\[
f''=-b(s)f,
\qquad f(0)=R,
\qquad f'(0)=1.
\]
Then $\lambda_0=mb>0$.  Put
\[
\psi=\frac{1-(f')^2}{f^2}.
\]
For small $s>0$, $f$ and $f'$ remain bounded and positive, while
\[
0<1-(f'(s))^2
=2\int_0^sb(\tau)f(\tau)f'(\tau)\dd\tau
\le Csb(s).
\]
Hence $\psi>0$ and $\psi/b\to0$.  Formula \eqref{eq:single-lambdaS} becomes
\[
\lambda_S=b+(m-1)\psi.
\]
After shrinking the collar, $0<\psi<b$, so
\[
b<\lambda_S<mb=\lambda_0.
\]
Therefore
\[
\dric=\frac{b+(m-1)\psi}{mb}>\frac1m,
\qquad
\dric\longrightarrow\frac1m
\quad(s\downarrow0).
\]

For a negative collar solve $f''=bf$ with the same flat initial data.  Then $\lambda_0=-mb$, and with
\[
\psi=\frac{(f')^2-1}{f^2}>0
\]
one has $\psi/b\to0$ and
\[
\lambda_S=-b-(m-1)\psi.
\]
The same pinching formula follows.

Finally, choose disjoint angular patches $U_+,U_-\subset\Sph^m$ and attach the positive and negative collars to disjoint portions of the boundary of a common Euclidean ball, exactly as in the proof of Theorem~\ref{thm:critical-local}.  The result is connected and smooth, with an open flat core.
\end{proof}

This proves the remaining parts of Theorem~\ref{thm:ricci-intro}.

\section{Global lower-bound examples and exactness obstructions}\label{sec:global-obstructions}

\subsection{Closed and complete subcritical examples}

\begin{proof}[Proof of Theorem~\ref{thm:global-intro}]
Let $m=n-1$, $a=a_\delta$, and
\[
f(t)=a-\cos t.
\]
The condition $0<\delta<1/m$ gives $a>m\ge2$, so $f>0$.  Put $c=\cos t$.  Since $f'=\sin t$ and $f''=c$, formulas \eqref{eq:single-lambda0}--\eqref{eq:single-lambdaS} give
\begin{align}
\lambda_0&=-\frac{mc}{a-c},\label{eq:compact-lambda0}\\
\lambda_S&=-\frac{c(a-mc)}{(a-c)^2}.
\label{eq:compact-lambdaS}
\end{align}
Because $a>m$, the factor $a-mc$ is positive for every $c\in[-1,1]$.  Thus both eigenvalues have sign $-\operatorname{sign}(c)$ and vanish simultaneously precisely when $c=0$.

At a non-Ricci-flat point,
\[
\abs{\lambda_0}=\frac{m\abs c}{a-c},
\qquad
\abs{\lambda_S}=\frac{\abs c(a-mc)}{(a-c)^2}.
\]
Moreover
\[
m(a-c)-(a-mc)=(m-1)a>0,
\]
so $\abs{\lambda_S}<\abs{\lambda_0}$.  Hence
\begin{equation}\label{eq:compact-pinching}
\dric(t)=\frac{a-m\cos t}{m(a-\cos t)}.
\end{equation}
As a function of $c=\cos t$,
\[
D(c)=\frac{a-mc}{m(a-c)},
\qquad
D'(c)=-\frac{a(m-1)}{m(a-c)^2}<0.
\]
The minimum occurs at $c=1$ and equals
\[
\frac{a-m}{m(a-1)}.
\]
Substitution of \eqref{eq:a-delta-intro} gives exactly $\delta$.  This proves the compact statement.

For the lift to $\R\times\Sph^m$, the warping function satisfies
\[
0<a-1\le a-\cos t\le a+1.
\]
The lifted metric is therefore uniformly equivalent to the complete product metric $\dd t^2+g_{\Sph^m}$, and is complete.  The curvature calculation is periodic and unchanged.
\end{proof}

\begin{remark}[Lower bound versus exact equality]\label{rem:not-exact}
Formula \eqref{eq:compact-pinching} varies with $t$.  Its minimum is the prescribed $\delta$, but its continuous limiting value at either Ricci-flat hypersurface is
\[
\frac1m=\delta_c.
\]
Thus Theorem~\ref{thm:global-intro} solves the compact and complete problem for the lower-bound condition, not for exact pointwise pinching.
\end{remark}

\subsection{A singly warped obstruction}

\begin{proposition}[An exact phase cannot emerge from a flat interface]\label{prop:single-obstruction}
Let $(F^m,g_F)$ have constant sectional curvature $\kappa$, let $m\ge2$, and consider
\[
g=\dd t^2+f(t)^2g_F,
\qquad f>0.
\]
Suppose a connected nonflat definite Ricci phase $I$ has exact pinching $\dric\equiv1/C$, with $C\ge m$, and that a point $t_0\in\overline I$ is Ricci-flat.  Then such a phase cannot meet $t_0$: the warping function is affine, and the metric is Ricci-flat, on a one-sided neighborhood of $t_0$ contained in $I$.
\end{proposition}

\begin{proof}
The Ricci eigenvalues are
\begin{align*}
\lambda_0&=-m\frac{f''}{f},\\
\lambda_F&=-\frac{f''}{f}+(m-1)\frac{\kappa-(f')^2}{f^2},
\end{align*}
with multiplicities $1$ and $m$.  Since the phase is definite and $C>1$, one of the two identities
\[
\lambda_0=C\lambda_F
\qquad\text{or}\qquad
\lambda_F=C\lambda_0
\]
holds throughout $I$; the ordering cannot switch without the ratio passing through $1$.

In the first case,
\begin{equation}\label{eq:single-obstruction1}
(C-m)ff''=C(m-1)\bigl(\kappa-(f')^2\bigr).
\end{equation}
At the Ricci-flat point,
\begin{equation}\label{eq:flat-data}
f''(t_0)=0,
\qquad
(f'(t_0))^2=\kappa.
\end{equation}
If $C>m$, equation \eqref{eq:single-obstruction1} is a smooth second-order ODE near $t_0$.  The affine function
\[
f_{\mathrm{flat}}(t)=f(t_0)+f'(t_0)(t-t_0)
\]
solves the same ODE with the same Cauchy data, so uniqueness forces $f=f_{\mathrm{flat}}$ on $I$ near $t_0$.  If $C=m$, equation \eqref{eq:single-obstruction1} gives $\kappa-(f')^2=0$ throughout the phase; differentiating, or treating separately $\kappa=0$, again gives $f''=0$ and flatness.

In the second ordering,
\begin{equation}\label{eq:single-obstruction2}
ff''=-\frac{m-1}{Cm-1}\bigl(\kappa-(f')^2\bigr).
\end{equation}
This is a smooth ODE for $f>0$, and the same affine solution has the Cauchy data \eqref{eq:flat-data}.  Uniqueness again forces Ricci-flatness on $I$ near $t_0$.
\end{proof}

In total dimension $n=m+1$, the condition $C\ge m$ is exactly $\delta=1/C\le1/(n-1)$.  Proposition~\ref{prop:single-obstruction} is an obstruction to one ansatz, not an unrestricted global nonexistence theorem.

\subsection{A periodic doubly warped obstruction}

\begin{proposition}[Fixed labeled relation plus periodicity forces flatness]\label{prop:periodic-obstruction}
Let $k\ge1$ and $C>1$.  On $\Sph^1\times\T^{k+1}$ consider
\[
g=\dd t^2+e^{2u(t)}\dd x^2+e^{2v(t)}\sum_{\alpha=1}^k\dd y_\alpha^2,
\]
where $u$ and $v$ are periodic.  If the Ricci eigenvalues satisfy
\begin{equation}\label{eq:periodic-relation}
\rho_0=C\rho_1
\end{equation}
identically, then $u$ and $v$ are constant and $g$ is flat.
\end{proposition}

\begin{proof}
Put
\[
A=u',
\qquad B=v',
\qquad
\alpha=\frac{k}{C-1}>0.
\]
Relation \eqref{eq:periodic-relation} is equivalent to
\[
D_A=\alpha D_B-C\alpha AB.
\]
Set $W=A-\alpha B$.  Substitution gives
\begin{equation}\label{eq:W-periodic}
W'+W^2+\nu BW+\mu B^2=0,
\end{equation}
where
\[
\nu=\alpha(C+2),
\qquad
\mu=\alpha\bigl((C+1)\alpha-1\bigr).
\]
Moreover
\[
(C+1)\alpha-1
=\frac{(k-1)C+k+1}{C-1}>0,
\]
so $\mu>0$.  Since $B=v'$, equation \eqref{eq:W-periodic} becomes
\[
\bigl(We^{\nu v}\bigr)'
=-e^{\nu v}(W^2+\mu B^2).
\]
Integrating over one period and using periodicity gives
\[
0=-\int_{\Sph^1}e^{\nu v}(W^2+\mu B^2)\dd t.
\]
Hence $W\equiv B\equiv0$, and then $A=W+\alpha B\equiv0$.
\end{proof}

This excludes the most direct periodization of the exact local model, namely one in which the same labeled Ricci eigendirections realize the extremal ratio everywhere.  A successful compact exact construction would have to change the extremal eigenbundle, use more warping functions, introduce singular cohomogeneity-one endpoints or monodromy, or leave the cohomogeneity-one setting.

\section{Dimension two, quarter pinching, and the status of the global exact problem}\label{sec:summary}

\subsection{Surfaces}

On a surface,
\[
\Ric=Kg.
\]
Therefore $\dric=1$ at every non-Ricci-flat point.  Exact Ricci pinching by any $\delta<1$ is impossible.  The metric \eqref{eq:surface-warped} gives a closed sign-changing exact-critical example, since both Ricci eigenvalues equal the sign-changing Gaussian curvature.

\subsection{Quarter-pinching consequences}

\begin{center}
\renewcommand{\arraystretch}{1.25}
\begin{tabularx}{\textwidth}{>{\raggedright\arraybackslash}p{0.16\textwidth}
>{\raggedright\arraybackslash}p{0.19\textwidth}X}
\toprule
Curvature & Dimension & Strongest conclusion established here \\
\midrule
Sectional & $n=2$ & The optimal pinching is $1$ at every nonflat point.  Exact $1/4$ pinching is impossible; the lower-bound convention allows sign change.\\
Sectional & $n\ge3$ & Any locally uniform positive pinching lower bound, in particular $\dsec\ge1/4$, forbids coexistence of positive and negative phases.\\
Ricci & $n=2$ & The optimal pinching is $1$ at every non-Ricci-flat point; exact $1/4$ pinching is impossible.\\
Ricci & $n=3,4$ & Local exact-$1/4$ sign-changing examples exist.  Closed sign-changing examples exist with $\dric\ge1/4$ and optimal global lower constant $1/4$.\\
Ricci & $n=5$ & The value $1/4$ is critical; local exact-$1/4$ sign-changing examples exist.\\
Ricci & $n\ge6$ & The lower bound $\dric\ge1/4$ is supercritical and forces one Ricci sign globally.\\
\bottomrule
\end{tabularx}
\end{center}

For the compact lower-bound examples in dimensions $3$ and $4$, formula \eqref{eq:a-delta-intro} gives respectively
\[
\dd t^2+(3-\cos t)^2g_{\Sph^2}
\quad\text{on }\Sph^1\times\Sph^2,
\]
and
\[
\dd t^2+(9-\cos t)^2g_{\Sph^3}
\quad\text{on }\Sph^1\times\Sph^3.
\]

\subsection{What remains open}

The results sharply distinguish exact pointwise equality from a lower pinching bound.  In dimensions $n\ge3$ they do not decide whether there is a closed or complete sign-changing metric with
\[
\dric\equiv\delta
\]
for a prescribed $0<\delta\le1/(n-1)$.  The local exact constructions show that no local obstruction is possible in this range, while Propositions~\ref{prop:single-obstruction} and~\ref{prop:periodic-obstruction} show that two natural globalizations fail.

The sectional theory leaves different questions.  The metric in Proposition~\ref{prop:conformal-signchange} is incomplete and has rapidly degenerating pinching.  The existence of complete or closed sectional-sign-changing examples satisfying the strict pointwise trichotomy but no uniform pinching remains outside the present arguments.  The constructed metric degenerates superpolynomially on one side and quartically on the other; determining the sharp possible rate for metrics is also unresolved by this work.  Remark~\ref{rem:linear-tensor} shows that linear degeneration is attainable for abstract divergence-free tensors.  Finally, for a fixed sectional pinching constant $\delta_0<1/2$, Theorem~\ref{thm:flat-intro} does not determine whether the flat set can have positive measure or Hausdorff dimension $n-1$.

\section{Concluding perspective}

A single divergence-free object, the Einstein tensor, organizes both theories.  For sectional curvature, the complementary-sum identity \eqref{eq:Ein-section} makes the appropriately signed Einstein tensor positive definite on every strictly one-signed phase.  Any locally uniform positive lower bound for the sectional pinching controls its condition number, and the support-rigidity theorem forbids a change of sign.

For Ricci curvature, positivity of the signed Einstein tensor is itself a pinching phenomenon.  The value $1/(n-1)$ is exactly where
\[
\frac12\Scal\,g-\Ric
\]
passes from uniform positivity relative to its trace to mere semidefiniteness.  Above this value, support rigidity gives global sign rigidity; at the endpoint a null direction can occur and exact local sign-changing metrics exist; below it, multiply warped metrics realize arbitrary exact ratios.  For the weaker lower-bound problem, explicit closed and complete subcritical examples persist.

The resulting picture is sharp at the local level and nearly complete globally: sectional sign change is impossible under every locally uniform positive pinching bound in dimensions $n\ge3$; the Ricci threshold is $1/(n-1)$, including critical and subcritical exact local counterexamples and failure of nonuniform pointwise supercriticality; closed and complete subcritical Ricci examples exist for the lower-bound condition; and two natural compactification ansatzes for exact Ricci pinching are obstructed.  The unrestricted closed or complete exact Ricci problem at and below the threshold remains open.

\end{document}